\documentclass[graybox,12pt]{svmult}

\usepackage{mathptmx}       % selects Times Roman as basic font
\usepackage{helvet}         % selects Helvetica as sans-serif font
\usepackage{courier}        % selects Courier as typewriter font
\usepackage{type1cm}        % activate if the above 3 fonts are
\usepackage{makeidx}         % allows index generation
\usepackage{graphicx}        % standard LaTeX graphics tool
\usepackage{multicol}        % used for the two-column index
\usepackage[bottom]{footmisc}% places footnotes at page bottom
\usepackage{amsmath}              % great math stuff
\usepackage{amsfonts}              % for blackboard bold, etc
\usepackage[ansinew]{inputenc}
\usepackage{array}
\usepackage{color}
\usepackage{amsxtra}
\usepackage{amstext}
\usepackage{amssymb}
\usepackage{latexsym}
\usepackage{dsfont}
\usepackage[all]{xy}
\usepackage{enumitem}  % needed for setting standard space after item labels of numbered lists itemize, enumerate, etc.
\usepackage{cite}

\usepackage{url}

\usepackage{xcolor}

\usepackage{verbatim}

\allowdisplaybreaks[3]

\title*{On $(\sigma,\tau)$-derivations of group algebra as category characters}
\titlerunning{On $(\sigma,\tau)$-derivations of group algebra as category characters}
\author{Aleksandr Alekseev \and Andronick Arutyunov \and Sergei Silvestrov }
\authorrunning{Aleksandr Alekseev, Andronick Arutyunov, Sergei Silvestrov}
\institute{Aleksandr Alekseev \at V.A. Trapeznikov Institute of Control Sciences of RAS, 65 Profsoyuznaya street, 117997, Moscow, Russia. \\
\email{aleksandr.alekseev@frtk.ru}
\and
Andronick Arutyunov \at
Department of Higher Mathematics, Moscow Institute of Physics and Technology, 9 Institutski
per., Dolgoprudny, 141701, Russia \\ \email{andronick.arutyunov@gmail.com}
\and
Sergei Silvestrov \at Division of Applied Mathematics,
School of Education, Culture and Communication,
M\"{a}lardalen University, Box 883, 72123 V{\"a}ster{\aa}s, Sweden. \\
\email{sergei.silvestrov@mdh.se}
}

\begin{document}

\maketitle
\label{Chap:AlekseevArutyunovSilvestrov}

\abstract*{For the space of $(\sigma,\tau)$-derivations of the group algebra $ \mathbb{C} [G] $ of discrete countable group $G$, the decomposition theorem  for the space of $(\sigma,\tau)$-derivations, generalising the corresponding theorem on ordinary derivations on group algebras, is established in an algebraic context using groupoids and characters.  Several corollaries and examples describing when all $(\sigma,\tau)$-derivations are inner are obtained. Considered in details cases on $(\sigma,\tau)-$nilpotent groups and $(\sigma,\tau)$-$FC$ groups.
\keywords{group algebra, derivation, $(\sigma,\tau)$-derivation, groupoid, character} \\
{\bf MSC2020 Classification:} 16W25, 13N15, 16S34}

\abstract{For the space of $(\sigma,\tau)$-derivations of the group algebra $ \mathbb{C} [G] $ of a discrete countable group $G$, the decomposition theorem  for the space of $(\sigma,\tau)$-derivations, generalising the corresponding theorem on ordinary derivations on group algebras, is established in an algebraic context using groupoids and characters.  Several corollaries and examples describing when all $(\sigma,\tau)$-derivations are inner are obtained. Considered in details cases on $(\sigma,\tau)-$nilpotent groups and $(\sigma,\tau)-$$FC$ groups.
\keywords{group algebra, derivation, $(\sigma,\tau)$-derivation, groupoid, character} \\
{\bf MSC2020 Classification:} 16W25, 13N15, 16S34}

\section{Introduction}
The general theory of derivations for $C^*$-algebras, $W^*$-algebras, Banach, normed and topological algebras, motivated by many parts of Mathematics and Mathematical Physics, has developed since 1950'th. Fundamental derivation theorems describing conditions for all or almost all derivations being inner, constructions of outer derivations and related cohomology methods and generalizations have been developed for $C^*$-algebras, $W^*$-algebras and some related classes of Banach, normed and topological algebras and their representations \cite{AEPT76,AP79,BHKS20,Bratteli86SLN1229derivgract,Dales,Elliot77,Helemskii8689HomolBanTopAlg,JS68,JR69,JP72,J72MAMSCohomBanachalg,Kad66,KadRing66,KadRing67,K53,K58,Ped19792018,P87,S60,S66,S68,S71,S91}.

For group algebras, the {\it derivation problem} is formulated as follows:  "Under what conditions all derivations in a group algebra are inner?". There are many kinds of group algebras based on algebraic structure, topological structure, measure structures and choices of function spaces for the group algebra elements. For the group algebra $L_1(G)$, the derivation problem is important for investigations in measure theory and harmonic analysis, operator theory, operator algebras and cohomological  constructions \cite{Dales}[question 5.6.B, p. 746]. The derivation problem for $L_1(G)$ of a locally compact group $G$ was considered in \cite{Losert}, where it was mentioned that all derivations of $L_1(G)$ are inner.

It is important to note that in the cited publications derivations are considered in topological context of the classes of algebras and modules equipped with normed or more general topological structures. If we consider the problem in algebraic way it is easy to find examples of non-inner derivations \cite{Ar,AMS16,AM}. Algebraic view on the derivation problem is also presented in \cite{ABS20} together with more complete bibliography.

We consider in this article $(\sigma,\tau)$-derivations, the linear operators on an associative algebra satisfying a generalized Leibniz rule $ D(xy)=D(x)\tau(y)+\sigma(x)D(y)$ twisted by two linear maps $\sigma$ and $\tau$.
The $(\sigma,\tau)$-derivation operators include, for example, the ordinary derivations on commutative and non-commutative algebras, the $q$-difference and $(p,q)$-difference operators on algebras of functions,
the super-derivations, graded colored derivations and $q$-derivation on graded associative algebras.
Since 1930's, $(\sigma,\tau)$-derivations and subclass of $\sigma$-derivations have been shown to play a fundamental role in the theory of Ore extensions and iterated Ore extension rings and algebras, skew polynomial algebras and skew fields, Noetherian rings and algebras, differential and difference algebra, homological algebra, Lie algebras and Lie groups, Lie superalgebras and colored Lie algebras, operator algebras, non-commutative geometry, quantum groups and quantum algebras, differential geometry, symbolic algebra computations and algorithms $q$-analysis and $q$-special functions and numerical analysis. \cite{BerIntroalgananticomvar,BoLvHarNoncomrings,Cohn65Differencealg,Cohn77Skewfieldsconstr,Jordan93,GoodWar89,Kharchenko8991AutDerivAssrings,KacCheung2002,Kasselbookqg99,LamLeroy88,Manin91,McConnelRobson2001,MRRSS90,Ore33,Rosenberg95}.
The space of $(\sigma,\tau)$-derivations and subspace of $\sigma$-derivations have been recently used in \cite{HartLarsSilv,LarsSilv:quasihomliecentext2cocycle,RichardSilvestrovJA2008,LarsSilv:quasidefsl2twistderiv,ElchLundMakhSilv,SongWuXin2015,SongXia2011,ChaudhuriComAlg2019}
in general constructions of quasi-Hom-Lie algebras and their central extensions, extending Witt and Virasoro Lie algebras to context of $(\sigma,\tau)$-derivations, with special emphasize on $(\sigma,\tau)$-derivations and $\sigma$-derivations on commutative algebras such as unique factorization domains, algebras of polynomials, Laurent polynomials and truncated algebras of polynomials in one or several variables.

In this paper, we consider $(\sigma,\tau)$-derivations and $\sigma$-derivations of the group algebra
$\mathbb{C} [G]$ of a discrete countable group $G$. We apply the approach proposed in \cite{AMS16,Ar,AM,Ar-Al,Mis} for the description of derivations in group algebras to study the space of $(\sigma,\tau)$-derivations of the group algebras.

Section \ref{sec:definprelim} contains general definitions and preliminaries on $(\sigma,\tau)$-derivations and group algebras considered in the article.
In Section \ref{sec:groupoidandcharacters}, we construct a groupoid $\Gamma$ associated with the group algebra and the pair of maps $(\sigma,\tau)$ in the $(\sigma,\tau)$-twisted Leibniz rule for $(\sigma,\tau)$-derivations, $(\sigma,\tau)$-conjugacy classes and characters on this groupoid.
We also construct an isomorphism between the space of $(\sigma,\tau)$-derivations on a group algebra of countable group and the space of locally finite characters on the associated groupoid (Theorem \ref{th-char-der}). In Section \ref{sec:quasiinnersigmatauder}, we define and describe quasi-inner $(\sigma,\tau)$-derivations, as well as a class of not quasi-inner derivations. We prove Theorem \ref{th-qinn-calc} which describes the view of quasi-inner $(\sigma,\tau)$-derivations. In Section \ref{sigma-tau-section} we consider the case of $(\sigma,\tau)$-nilpotent groups and obtain a description of the $(\sigma,\tau)$-derivation algebra (see Theorem \ref{th-decomposition}). For the case of inner endomorphisms and a Heisenberg group, we calculate the $(\sigma,\tau)$-derivations in the group algebra (see results in subsection \ref{subsec-Heis}). The Section \ref{sec:sigmatauFCgroups} is dedicated to the case of $(\sigma,\tau)-$$FC$ groups which are natural generalization of usual $FC$-groups on our ''twisted'' case of $(\sigma,\tau)$-derivations. For this class of groups we show some simple properties (see Proposition \ref{prop-fc-sigmasigma}) and construct a description of $(\sigma,\tau)$-derivations (Theorem \ref{th-FC}).

%%%%%%%%%%%%%%%%%%%%%%%%%%%%%%%%%

\section{General definitions and preliminaries} \label{sec:definprelim}
In this section, we recall some basic definitions and properties of main objects studied in the rest of the article.
\begin{definition}
Let $\mathcal{A}$ be an associative algebra over a field $\mathcal{K}$ and $(\sigma,\tau)$ is a pair of $\mathcal{K}$-linear endomorphisms of $\mathcal{A}$.
A \emph{$(\sigma,\tau)$-derivation} $D : \mathcal{A} \to \mathcal{A}$ is an $\mathcal{K}$-linear map such that the following twisted by $(\sigma,\tau)$ generalized Leibniz identity
\begin{equation}
D(ab)=D(a)\tau(b) + \sigma(a)D(b) \label{sigmatauderLeibniz}
\end{equation}
is satisfied for all $a,b \in \mathcal{A}$.
If $\tau=id_{\mathcal{A}}$, the generalized Leibniz identity is twisted by one map $\sigma$ as follows
\begin{equation}
D(ab)=D(a)b + \sigma(a)D(b), \label{sigmaderLeibniz}
\end{equation}
and the $(\sigma,id_{\mathcal{A}})$-derivation $D$ is called \emph{$\sigma$-derivation}.
If $\tau=\sigma=id_{\mathcal{A}}$, then the usual Leibniz identity
\begin{equation}
D(ab)=D(a)b + aD(b), \label{derLeibniz}
\end{equation}
holds, and $D$ is called a derivation on $\mathcal{A}$.

The set of all $(\sigma,\tau)$-derivations on $\mathcal{A}$ is denoted by $\mathcal{D}_{(\sigma,\tau)}(\mathcal{A})$.
The set of all $\sigma$-derivations is denoted by $\mathcal{D}_{\sigma}(\mathcal{A})$, and the set of all derivations is denoted by $\mathcal{D}(\mathcal{A})$
\end{definition}

The set $\mathcal{D}_{(\sigma,\tau)}(\mathcal{A})$ of $(\sigma,\tau)$-derivations on $\mathcal{A}$ is a linear subspace of the space of linear operators on $\mathcal{A}$ as a $\mathcal{K}$-linear space.

For any $(\sigma,\tau)$-derivation on $\mathcal{A}$,
\begin{multline*}
D(a(bc)-(ab)c)=D(a(bc))-D((ab)c)=\\
=D(a)\tau(bc)+\sigma(a)D(bc)-\big(D(ab)\tau(c)+\sigma(ab)D(c)\big)=\\
=D(a)\tau(bc)+\sigma(a)\big(D(b)\tau(c)+\sigma(b)D(c)\big)\\
\quad-\big(D(a)\tau(b)+\sigma(a)D(b)\big)\tau(c)-\sigma(ab)D(c)=\\
=D(a)\big(\tau(bc)-\tau(b)\tau(c)\big)+\big(\sigma(a)\sigma(b)-\sigma(ab)\big)D(c).
\end{multline*}
Since $\mathcal{A}$ is associative, $a(bc)-(ab)c=0$, and since $D$ is linear, $D(a(bc)-(ab)c) = 0$,
and hence
\begin{equation} \label{assocsigmatauderiv}
D(a)\big(\tau(bc)-\tau(b)\tau(c)\big)+\big(\sigma(a)\sigma(b)-\sigma(ab)\big)D(c) = 0
\end{equation}
Note that for the identity element $e$, in general,
$
D(e) = D(e \cdot e) = D(e)\tau(e) + \sigma(e)D(e),
$
or equivalently,
$
D(e)(e-\tau(e)) = \sigma(e)D(e),
$
implying that if $\tau(e)=\sigma(e)= e $ (as for example for group endomorphisms), then $D(e) = 0$.

In this article we will be interested only in $(\sigma,\tau)$-derivations $D$ with $\sigma, \tau \in \textbf{End}(G)$ and $D (e) = 0$.

If $\mathcal{A}$ is an associative algebra over a field $\mathcal{K}$,
and $\sigma$ and $\tau$ are algebra endomorphisms on $\mathcal{A}$, then
for any $p \in \mathcal{A}$ the map $\delta_p: \mathcal{A} \rightarrow \mathcal{A}$ given by the $(\sigma,\tau)$-twisted generalised commutator $\delta_p(x)= p\tau(x) - \sigma(x)p$, is a $(\sigma,\tau)$-derivation on $\mathcal{A}$, since for each $x, y \in \mathcal{A}$, we have
\begin{multline*}
\delta_p(xy) = p\tau(xy) - \sigma(xy)p = p\tau(x)\tau(y) - \sigma(x)\sigma(y)p = \\
= p\tau(x)\tau(y) - \sigma(x)p\tau(y) + \sigma(x)p\tau(y) - \sigma(x)\sigma(y)p = \\
= \delta_p(x)\tau (y) + \sigma(x)\delta_p(y).
\end{multline*}
\begin{definition}[inner $(\sigma,\tau)$-derivation]
If $\mathcal{A}$ is an associative algebra, and $\sigma$ and $\tau$ are algebra endomorphisms on $\mathcal{A}$, then $(\sigma,\tau)$-derivations $\delta_p$ for $p \in \mathcal{A}$ are called the inner $(\sigma,\tau)$-derivations of $\mathcal{A}$.
\end{definition}

Throughout this article, the group algebra $\mathbb{K} [G]$ of a group $(G,\cdot)$
over the field $\mathbb{K}$ means the linear space of mappings
$f: G \to \mathbb{K}$ of finite support with the pointwise operations of multiplication by scalars and addition, and the algebra product defined as convolution
\begin{equation}
(f*g)(x)= \sum _{u\cdot v=x}f(u)g(v)=\sum _{u\in G}f(u)g(u^{-1}x).
\end{equation}
where all sums are finite because f and g are of finite support.
With these operations, the group algebra is an unital associative algebra with the algebra unity coinciding with the indicator function $I_e$ of the group unity $e\in G$, that is $I(e)= 1_{\mathbb{K}}$ and $I(e)=0_{\mathbb{K}}$ otherwise.
The elements $f\in \mathbb{K} [G]$ often are conveniently presented as the formal linear combinations of elements of $G$ with coefficients in $\mathbb{K}$ written as $\sum_{g\in G}f(g)g$ or $\sum_{g\in G}f_{g}g$ similar to usual way of writing polynomials and Laurent polynomials.

\section{Groupoid and characters}
\label{sec:groupoidandcharacters}
For any mapping $D: \mathbb{C}[G]\rightarrow \mathbb{C} [G]$ and any $g\in G$,
the element $D(g) \in \mathbb{C} [G]$ can be written as
$$
D(g) = \sum\limits_{h \in G} \lambda^{h}_g h, \quad \quad \lambda^{h}_g \in \mathbb{C}.$$
When $D:\mathbb{C} [G]\rightarrow \mathbb{C} [G]$ is $(\sigma,\tau)$-derivation,
the $(\sigma,\tau)$-twisted Leibniz rule \eqref{sigmatauderLeibniz} for
$g_1,g_2 \in G \hookrightarrow \mathbb{C} [G]$ becomes
\begin{align}
D(g_2g_1) &= D(g_2)\tau(g_1) + \sigma(g_2)D(g_1), \nonumber \\
\label{Leibnizcomponent}
\lambda^h_{g_2g_1} &= \lambda^{h\tau(g_1^{-1})}_{g_2} + \lambda^{\sigma(g_2^{-1})h}_{g_1}.
\end{align}

Here we apply an approach developed in \cite{AMS16,AM,Mis,Ar-Al} to describe derivation in terms of groupoid geometrical properties. We construct a groupoid $\Gamma$ associated with the group algebra in the following way:

\begin{itemize}
\item $Obj(\Gamma) = G$
\item For all $a$ and $b \in Obj(\Gamma)$ a set of maps is \textbf{Hom} $(a, b) = \{(u, v)\in G \times G \, | \, \sigma(v^{-1})u = a, u\tau(v^{-1}) = b\}$
\item A composition of maps $\varphi = (u_1, v_1) \in$ \textbf{Hom} $(a, b)$, $\psi = (u_2, v_2) \in$ \textbf{Hom} $(b, c)$ is a map $\varphi \circ \psi \in$ \textbf{Hom} $(a, c)$, such that
\begin{equation}
\label{eq-composition-morph}
	\varphi \circ \psi = (u_2\tau(v_1), v_2v_1).
\end{equation}
\end{itemize}

\begin{remark}
	We do not have in formula \eqref{eq-composition-morph} the map $\sigma$ explicitly. But if the composition $\varphi\circ\psi$ exists, then $t(\psi) = s(\varphi)$. This means that
	$$
	 \sigma(v_1^{-1})u_1 = u_2\tau(v_2^{-1}).
	$$
	The last formula gives us a connection to map $\sigma$ in \eqref{eq-composition-morph}.
\end{remark}

We will be interested in some internal structure issues of the groupoid $\Gamma$. If morphism $(u,v)\in Hom(a,b)$, then $u=\sigma(v)a$ and $u\tau(v^{-1})=b$. Therefore we have
 $$
	b = \sigma(v) a \tau(v^{-1}).
$$
It gives the following description of subgroupoid $\Gamma_{[a]}$
\begin{equation}
	Obj (\Gamma_{[a]}) = \{\sigma(v) a \tau(v^{-1})|v\in G\}.
\end{equation}

\begin{definition}
\label{def-st-conj-class}
	A subset $[a]_{\sigma,\tau} = \{\sigma(g^{-1})a\tau(g) \, | \, g \in G\}$ is called a $(\sigma,\tau)$-conjugation class of the $u$.
	
\end{definition}

The set of $(\sigma,\tau)$-conjugation classes is denoted $G^{(\sigma,\tau)}$. If two $(\sigma,\tau)-$conjugacy classes intersect, then they coincide. So the set of group elements is represented as a disjoint union of $(\sigma,\tau)-$conjugacy classes
$$
	\{G\} =  \underset{[u]_{\sigma,\tau}\in G^{(\sigma,\tau)}}{\bigsqcup} {[u]_{\sigma,\tau}}.
$$

The set $[a]_{(\sigma,\tau)}:=\{\sigma(v) a \tau(v^{-1})\,|\, v\in G\}$ can be understood as $(\sigma,\tau)$ analogue of conjugacy class. If the $(\sigma,\tau)$-class $[a]$ contains just one element, then we will say that $a$ is a $(\sigma,\tau)$-central element of $G$.

It is easy to see that $\Gamma$ as disjoint union of the groupoids $\Gamma_{[u]_{\sigma,\tau}}$ (see \cite{AMS16}, \cite{Ar}):
\begin{equation}
\label{eq-Gamma-decomp}
	\Gamma = \underset{[u]_{\sigma,\tau}\in G^{(\sigma,\tau)}}{\bigsqcup} \Gamma_{[u]_{\sigma,\tau}}
\end{equation}

\begin{definition}
	A linear map $\chi : \textbf{Hom} (\Gamma) \rightarrow \mathbb{C}$, such that
	\begin{equation} \label{characterproperty}
	\chi(\varphi \circ \psi) = \chi(\varphi) + \chi(\psi),
	\end{equation}
	is called a character on the groupoid $\Gamma$.
\end{definition}

Characters with natural sum operations and multiplication by scalar form a vector space. It's naturally to define the support of character $\chi$ in following way
\begin{equation}
	supp \chi = \{\varphi\in Hom(\Gamma)|\chi(\varphi)\neq 0 \}.
\end{equation}

 Further we will be interested just in those characters which are connected with derivations.

\begin{definition}
	The character $\chi$ such that, for fixed $v \in Obj(C) \,$, $\chi(u, v) = 0$ for almost all $u \in Obj(C)$ is called a locally finite character.
\end{definition}

Let $X(\Gamma)$ be a space of the all locally finite characters on $\Gamma$. From formula \eqref{eq-Gamma-decomp} we get the decomposition of the space $X(\Gamma)$ in following way
\begin{equation}
\label{eq-decomp-char-supports}
	X(\Gamma) = \bigoplus\limits_{[u]_{\sigma,\tau}\in G^{(\sigma,\tau)}} X(\Gamma_{[u]_{\sigma,\tau}}),
\end{equation}
where $X(\Gamma_{[u]_{\sigma,\tau}})$ denotes the locally finite characters supported in $\Gamma_{[u]_{\sigma,\tau}}$.

Consider a map $\Psi : \mathcal{D}_{(\sigma,\tau)}({\mathbb{C}[G]}) \rightarrow X(\Gamma)$, such that if $D(g) = \sum\limits_{h \in G} \lambda^{h}_g h$, then $ \Psi(D)(h, g) = \lambda^{h}_g$. The map $\Psi^{-1}$ is constructed in the same way, $$\Psi^{-1}(\chi)(g) = \sum\limits_{h \in G} \chi(h, g) h.$$

\begin{theorem}
\label{th-char-der}
		Consider the discrete countable group $G$ with $\sigma, \tau \in \textbf{End}(G)$. Then, the map $\Psi : \mathcal{D}_{(\sigma,\tau)}({\mathbb{C}[G]}) \rightarrow X(\Gamma)$ is an isomorphism.
\end{theorem}
\begin{proof}
		On the one hand, we show that $\Psi(D) \in X(\Gamma)$. Due to the definition of the groupoid $\Gamma$, there exists a following composition of maps:
		$$
		(h\tau(g_1^{-1}), g_2) \circ (\sigma(g_2^{-1})h, g_1) = (h, g_2g_1).
		$$
		Using the equation \eqref{Leibnizcomponent} one obtain:
		\begin{multline}
		\Psi(D)(h, g_2g_1) = \lambda^h_{g_2g_1} = \\
		= \lambda^{h\tau(g_1^{-1})}_{g_2} + \lambda^{\sigma(g_2^{-1})h}_{g_1} =  \Psi(D)(h\tau(g_1^{-1}), g_2) + \Psi(D)(\sigma(g_2^{-1})h, g_1).
		\end{multline}
		The latter equation means that $\Psi(D)$ satisfies the property \eqref{characterproperty}. Thus, $\Psi(D) \in X(\Gamma)$.
		
		On the other hand, due to the property of locally finiteness,
		$$
		\Psi^{-1}(\chi)(g) = \sum\limits_{h \in G} \chi(h, g) h \in \mathbb{C}[G],
		$$
		and $\Psi \Psi^{-1} = \operatorname{Id}_{X(\Gamma)}$, $\Psi^{-1}\Psi = \operatorname{Id}_{\mathcal{D}_{(\sigma,\tau)}({\mathbb{C}[G]})}$.
\end{proof}

\section{Quasi-inner ($\sigma,\tau$)-derivations}
\label{sec:quasiinnersigmatauder}

Recall that inner $(\sigma,\tau)$-derivation $\delta_p$ is given by formula
$$
	\delta_p: x\mapsto p\tau(x) - \sigma(x) p.
$$
The corresponding by Theorem \ref{th-char-der} character is trivial on loops.

\begin{proposition}
	For the given inner $(\sigma,\tau)$-derivation $\delta_p$ a corresponding character $\Psi(\delta_p)$ is trivial on loops, in the other words $\forall a \in Obj(C)$ and $\forall \varphi \in \textbf{Hom} (a, a)$ the value $\Psi(\delta_p)(\varphi) = 0$.
\end{proposition}

\begin{proof}
	If $p\tau(g) = \sigma(g)p$, then $\Psi(\delta_p)(pg, g) = \lambda^{p\tau(g)}_{g} - \lambda^{\sigma(g)p}_{g} = 1 - 1 = 0$. Otherwise, if $\varphi \in \textbf{Hom} (p, p)$ then $\varphi = (p\tau(g), p) = (\sigma(g)p, p)$.
\end{proof}

Not all locally finite characters which are trivial on loops are given by an inner derivation even for ordinary derivations (see \cite[page 76 (example)]{AMS16}).

\begin{definition}
A $(\sigma,\tau)$-derivation $D$ is said to be a {\it quasi-inner} if the corresponding character $\Psi(D)$ is trivial on loops.
\end{definition}

When mapping $\sigma$ and $\tau$ are identity mappings $(\sigma,\tau)$-derivations are equal to usual derivations on group algebra $\mathbb{C}[G]$. In this case quasi-inner derivations form an ideal which contains ordinary inner derivations
(\cite[Theorem 4.1]{Ar-Al}), and quasi-inner derivations can be easily calculated (see Theorem \ref{th-qinn-calc}).

\begin{definition} \label{centralelement}
An element $a \in G$ is said to be $(\sigma,\tau)$-central if $a\tau(v) = \sigma(v)a \, \, \forall v \in G$.
\end{definition}

\begin{proposition}
	For the given group $G$, maps $\sigma, \tau \in \textbf{End}(G)$, $(\sigma,\tau)$-central element $a$ and a homomorphism $\varphi : G \rightarrow \mathbb{C}$, a map $D(g) = \varphi(g)\sigma(g)a$ is a $(\sigma,\tau)$-derivation.
\end{proposition}
	
\begin{proof}
Using first the definition of $D$, then $(\sigma,\tau)$-centrality of the element $a$ and then homomorphism property of $\varphi$ and $\sigma$ and finally again the definition of $D$ one obtains:
		\begin{multline*}
			D(g_1)\tau(g_2) + \sigma(g_1)D(g_2) =
		 \varphi(g_1)\sigma(g_1)a\tau(g_2) + \sigma(g_1)\varphi(g_2)\sigma(g_2)a =  \\
\varphi(g_1)\sigma(g_1)\sigma(g_2)a + \varphi(g_2)\sigma(g_1)\sigma(g_2)a = \\
(\varphi(g_1) + \varphi(g_2))\sigma(g_1)\sigma(g_2)a =
			 \varphi(g_1g_2)\sigma(g_1g_2)a = D(g_1g_2).		
\end{multline*}
which means that $D$ is a $(\sigma,\tau)$-derivation.
\end{proof}

\begin{definition}
\label{def-sigma-tau-central-der}
	The $(\sigma,\tau)$-derivation $D$ is called a $(\sigma,\tau)$-central.
\end{definition}

 The $(\sigma,\tau)$-central derivations give us an example of noninner $(\sigma,\tau)-$derivations. For the case of general group algebra they were studied in \cite{Ar}.

\begin{proposition}
	The nonzero $(\sigma,\tau)$-central $(\sigma,\tau)$-derivation $D$ is not quasi-inner.
\end{proposition}

\begin{proof}
	The character's value $\Psi(D)(\sigma(g)a, g) = \varphi(g)$, and since $(\sigma(g)a, g)$ is a loop, the $(\sigma,\tau)$-derivation $D$ is not quasi-inner.
\end{proof}

Quasi-inner $(\sigma,\tau)$-derivations can be calculated in the following way.
Consider a map $(u, v) \in \textbf{Hom}(\Gamma)$. Let $t, s : \textbf{Hom}(\Gamma) \rightarrow Obj(\Gamma)$ be a target and source maps, such that $\phi = (u, v) : s(\phi) \rightarrow t(\phi)$. If the character $\chi$ is trivial on loops then exists function $P_{\chi}:\operatorname{Obj}(\Gamma)\to \mathbb{C}$ such that
$$
	\chi (\phi) = P_{\chi}(t(\phi)) - P_{\chi}(s(\phi)).
$$

That means that if $\chi$ is locally finite character, then the following formula is valid for a quasi-inner $(\sigma,\tau)$-derivation $D_P$:
$$
D_P(g) = \sum\limits_{h \in G} (P_{\chi}(t(h, g)) - P_{\chi}(s(h, g)))h = \sum\limits_{h \in G} (P_{\chi}(h\tau(g^{-1})) - P_{\chi}(\sigma(g^{-1})h))h.
$$
	
In other words we get the following statement.

\begin{theorem}
\label{th-qinn-calc}
	If $D\in QInn(\Gamma)$, then exists a finitely supported function $P:Obj(\Gamma)\to \mathbb{C}$ such that for generators $g\in \mathbb{C}[G]$,
	$$
		D(g) = \sum\limits_{h \in G} (P(h\tau(g^{-1})) - P(\sigma(g^{-1})h))h.
	$$
\end{theorem}	

\begin{remark}
The function $P$ from the theorem is not unique. It is determined up to the addition of a constant on each subgroupoid.
\end{remark}

\section{$(\sigma,\tau)-$nilpotent groups}
\label{sigma-tau-section}

\subsection{General case of $(\sigma,\tau)-$nilpotent groups}
The following concepts are similar to the terms from classic group theory. Derivations in classic rank $2$ nilpotent groups were studied in \cite{Ar}.

\begin{definition}
	A subgroup $Z_{\sigma,\tau} = \{z \in G \, | \, \sigma(z)p = p\tau(z) \, \forall p \in G\}$ is called a $(\sigma,\tau)$-center of the group $G$.
\end{definition}

Of course $z\in Z_{\sigma,\tau}$ is equivalent to the fact that there is single object in subgroupoid $\Gamma_{[z]}$.

It is worth mentioning that the  $(\sigma,\tau)$-center is \textit{not} a subgroup of $(\sigma,\tau)$-central elements, which were introduced in Definition \ref{centralelement}.

\begin{definition}
	A subgroup $Z_{\sigma,\tau}(u) = \{z \in G \, | \, \sigma(z)u = u\tau(z)\}$, $u \in G$ is called a $(\sigma,\tau)$-centraliser of the element $u$.
\end{definition}

\begin{proposition}
	A subgroup $Z_{\sigma,\tau} \subset G$ is a normal subgroup.
\end{proposition}
\begin{proof}
	Let us show that for $z \in Z_{\sigma,\tau}$ and $g \in G$ an element $gzg^{-1} \in Z_{\sigma,\tau}$ or in the other words $p\tau(gzg^{-1}) = \sigma(gzg^{-1})p \, \, \forall p \in G$. Indeed,
\begin{align*} 	
& p\tau(gzg^{-1}) = p\sigma(z) = \sigma(\sigma^{-1}(p)z) = \\
& \quad \sigma(\tau^{-1}(\tau(\sigma^{-1}(p))\tau(z))) = \sigma(\tau^{-1}(\sigma(z)\tau(\sigma^{-1}(p))))  = \sigma(\tau^{-1}(\sigma(z)))p,\\
& \sigma(gzg^{-1})p = \sigma(\tau^{-1}(\sigma(z)))\sigma(gg^{-1})p = \sigma(\tau^{-1}(\sigma(z)))p.
\end{align*}
	
\end{proof}

In accordance with the definition of $\Gamma$ the source of the map $(\sigma(g)p, g)$ is the object $p$, the target is $\sigma(g)p\tau(g^{-1})$ and the inverse map $(\sigma(g)p, g)^{-1}= (p\tau(g^{-1}), g^{-1})$. That means that $g^{-1}Z_{\sigma,\tau}(p)g = Z_{\sigma,\tau}(\sigma(g)p\tau(g^{-1}))$.

\begin{definition}
	A group $G$ such that $G / Z_{\sigma,\tau}$ is abelian is called a $(\sigma,\tau)$-nilpotent group with rank $2$.
\end{definition}

\begin{proposition}
	Consider the $(\sigma,\tau)$-nilpotent group $G$ with rank $2$. Then all elements in $Obj(\Gamma_{[u]_{\sigma,\tau}})$ have the same $(\sigma,\tau)$-centralizer group or, in the other words, $Z_{\sigma,\tau}(\sigma(g)p\tau(g^{-1})) = gZ_{\sigma,\tau}(p)g^{-1} = Z_{\sigma,\tau}(p)$.
\end{proposition}
\begin{proof}
	Consider an element $z_p \in Z_{\sigma,\tau}(p)$. Let $[z_p]$ be a class in quotient group $G/Z_{\sigma,\tau}$. Due to the fact, that $G/Z_{\sigma,\tau}$ is abelian, one obtain, that $[gz_pg^{-1}] = [z_p]$. Thus, there exists an element $z \in Z_{\sigma,\tau}$, such that $z_pz = gz_pg^{-1}$. That means, that $gz_pg^{-1} \in Z_{\sigma,\tau}(p)$.
\end{proof}

\begin{corollary}
\label{coroll-1}
	If there is a subgroupoid $\Gamma_{[u]_{\sigma,\tau}}$ with the infinite number of objects, then each character $\chi \in X(\Gamma_{[u]_{\sigma,\tau}})$ is trivial on loops.
\end{corollary}
\begin{proof}
	The proof is immediately  follows from the statement that for the given $(\sigma,\tau)$-derivation $D$ the corresponding character $\Psi(D)$ has to be locally-finite. Consider an object $a \in Obj(\Gamma_{[u]_{\sigma,\tau}})$. Then, since
$$
Z_{\sigma,\tau}(\sigma(g)a\tau(g^{-1})) = Z_{\sigma,\tau}(a),
$$
if there is a loop
$$
(a\tau(g), g) \in \textbf{Hom}(a, a),
$$
then every set $\textbf{Hom}(b, b) \, \, \forall b \in Obj(\Gamma_{[u]_{\sigma,\tau}})$ has a map
$$
(b\tau(g), g) \in \textbf{Hom}(b, b).
$$
Due to the fact that
$$
\chi(a\tau(g), g) = \chi(b\tau(g), g),
$$
one obtains that if $\chi(a\tau(g), g) \neq 0$, then a character $\chi$ does not satisfy the property of locally-finiteness.
\end{proof}

We are going now to generalize the  result of the corollary as a following theorem.

\begin{theorem}
\label{th-decomposition}
If $G$ is rank $2$ $(\sigma, \tau)$-nilpotent group then
\begin{equation}
\label{eq-decomposition-theorem}
	D_{(\sigma,\tau)} \,  \cong \bigoplus\limits_{\#[a]_{(\sigma,\tau)}<\infty} Z^*_{(\sigma,\tau)}(a) \, \bigoplus \, QInn(\Gamma),
\end{equation}
where $Z^*_{(\sigma,\tau)}$ is the space of group characters of the centralizer $Z_{(\sigma,\tau)}(a)$, i.e. $Z^*_{(\sigma,\tau)} = \textbf{Hom}(Z_{(\sigma,\tau)}, \mathbb{C})$, and $QInn(\Gamma)$ is a space of the all quasi-inner derivations on $\Gamma$.
\end{theorem}

\begin{proof}
	As we noted above
	$$
		\Gamma = \underset{}{\bigsqcup} \Gamma_{[u]_{\sigma,\tau}}.
	$$
This implies the following decomposition
$$
	X(\Gamma) = \bigoplus X(\Gamma_{[u]_{\sigma,\tau}}),
$$
where $X(\Gamma_{[u]_{\sigma,\tau}})$ is the denotation of locally finite characters supported in a subgroupoid $\Gamma_{[u]_{\sigma,\tau}}$. If the class $[a]_{(\sigma,\tau)}$ is infinite, then by Corollary \ref{coroll-1}  characters from the subspace $X(\Gamma_{[u]_{\sigma,\tau}})$ are quasi-inner.

Now consider finite class $[u]_{(\sigma,\tau)}$. In $\Gamma_{[u]_{\sigma,\tau}}$ each set of maps $\{(a\sigma(g), g) \, | \, a \in Obj(\Gamma_{[u]_{\sigma,\tau}})\}$ is finite because the set of objects is finite. That means that each character on this subgroupoid is locally finite. Each equivalence class of the characters which have equals values on loops, i.e. $\chi_1-\chi_2\in QInn$, is defined by an element of the group $Z^*_{(\sigma,\tau)}(u)$.

%Now consider finite class $[a]_{(\sigma,\tau)}$. For each $g$ the set of morphisms in the form ${(*, g) \, | \, g \in G}$ which have the form $(\ast, g)$ is finite because the set of objects is finite. That means that each character on this subgroupoid is locally finite. Every equivalence class of the characters which are have equals values on loops, i.e $\chi_1-\chi_2\in QInn$ is defined by a group character from $Z^*_{(\sigma,\tau)}(u)$. We fix character $\chi_{1,2}\in X(\Gamma_{[u]_{\sigma,\tau}})$ and corresponding $(\sigma,\tau)$-derivations $d_1, d_2$. It's easy to see that $d_1-d_2\in QInn$ is equal to character $\chi_1-\chi_2$ is trivial on loops. Characters that are non-trivial on loops are uniquely defined by a group character from $Z^*_{(\sigma,\tau)}(u)$.

Whence we get that
$$
	D_{(\sigma,\tau)}(\Gamma_{[u]_{\sigma,\tau}}) \cong Z^*_{(\sigma,\tau)}(u) \bigoplus QInn(\Gamma_{[u]_{\sigma,\tau}}).
$$

From the above it follows the statement of the theorem.
\end{proof}

\begin{remark}
Remind that if $u\in Z_{\sigma,\tau}$ then the space $QInn(\Gamma_{[u]_{\sigma,\tau}})$ is trivial.
\end{remark}

Group $(\sigma,\tau)$-nilpotency is necessary for triviality of characters on loops on infinite subgroupoids. If our group $G$ is not nilpotent then the right side of \eqref{eq-decomposition-theorem} is just a subspace in the space of all $(\sigma,\tau)$-derivations.

Following to the papers cited in introduction of this article we will describe conditions of quasi-innerness of $(\sigma,\tau)$-derivations.

\begin{corollary}
\label{coroll-criterion-triv}
	For groups satisfying the conditions of the Theorem \ref{th-decomposition}, all $(\sigma,\tau)$-derivations are quasi-inner if and only if the following condition is satisfied: all $(\sigma,\tau)-$centralizers are such that factor-group $Z_{(\sigma,\tau)}(a)/ Z'_{(\sigma,\tau)}(a)$ is a periodic group.
\end{corollary}

Here $Z'_{(\sigma,\tau)}(a)$ is a commutant of group $Z_{(\sigma,\tau)}(a)$. Recall that a periodic group is a group such that all elements have finite order.

\begin{proof}
	The group $Z_{(\sigma,\tau)}(a)/Z'_{(\sigma,\tau)}(a)$ is naturally abelian. The periodicity of a group is equivalent to triviality of the space of group characters $Z_{(\sigma,\tau)}(a)/Z'_{(\sigma,\tau)}(a)$. Triviality for each $a\in G$ of spaces $Z^*_{(\sigma,\tau)}(a)$ -- is equal to triviality of first term in \eqref{eq-FC-decomposition-theorem}.
	So from the Theorem \ref{th-decomposition} we get that all $(\sigma,\tau)$-derivations are quasi-inner.
\end{proof}

\subsection{The case of inner endomorphisms}
\label{subsec:inneendoms}
	
	Let $G$ be a discrete rank $2$ nilpotent group and $\sigma, \tau \in \textbf{Aut}(G)$ act for fixed elements  $\tilde{\sigma},\tilde{\tau} \in G$ as follows
	\begin{center}
		$\sigma(u) = \tilde{\sigma} u \tilde{\sigma} ^{-1}$, $\tau(u) = \tilde{\tau} u \tilde{\tau}^{-1}$.
	\end{center}
	Let $Z(G)$ be the usual center of $G$.
	\begin{remark}
		We remind, that the group $G$ is said to be  rank $2$ nilpotent group, if and only if the quotient group $G/Z(G)$ is abelian. In our notation nilpotent rank $2$ group is a rank $2$ $(id,id)$-nilpotent group.
	\end{remark}
	\begin{proposition}
	\label{prop-center-sigma-tau}
		The usual center of $G$ is equal to $Z_{\sigma,\tau}$.
	\end{proposition}
	\begin{proof}
		Consider $z \in Z(G)$. Then
		$$
		\sigma(z)a = \tilde{\sigma} z \tilde{\sigma}^{-1}a = za = az = a \tilde{\tau} z \tilde{\tau}^{-1} = a\tau(z).
		$$
		Thus, $Z(G) \subset Z_{\sigma,\tau}$. Now consider $z \in Z_{\sigma,\tau}$. Then
		$$
			\sigma(z)a = a\tau(z) \rightarrow z  = \tilde{\sigma}^{-1} a \tilde{\tau} z \tilde{\tau}^{-1}a^{-1}\tilde{\sigma}
		$$
		Since $\tilde{\sigma} G \tilde{\tau}^{-1} = G$ as sets, the latter equation holds for every $g \in G$. Thus, $Z(G) = Z_{\sigma,\tau}$.
	\end{proof}
	
	\begin{corollary}
		The discrete rank $2$ nilpotent group $G$ coupled with $(\sigma, \tau)$ pair is a rank $2$ $(\sigma,\tau)$-nilpotent group.
	\end{corollary}
	\begin{proof}
		As we mentioned before, a quotient group $G / Z(G)$ is abelian and $Z(G) = Z_{\sigma,\tau}$. Thus, the given group is a rank $2$ $(\sigma,\tau)$-nilpotent group.
	\end{proof}

\subsection{Heisenberg group}
\label{subsec-Heis}
Our results and observations allow us to calculate all $(\sigma,\tau)$-derivations in Heisenberg group. In the calculation we will use results in \cite{Ar}[Section 3.3]. Recall that the Heisenberg group $H$ is a group of unitriangular integer matrices. We denote the group algebra as $\mathcal{H}$.

Classes $[u]_{\sigma,\tau}$ in Heisenberg group either consist of one element or infinite. That means that by Theorem \ref{th-decomposition} $$D_{\sigma,\tau}(\mathcal{H})= ZDer_{\sigma,\tau}\bigoplus QInn,$$
where $ZDer_{\sigma,\tau}$ denotes $(\sigma,\tau)$-central $(\sigma,\tau)$-derivations from Definition \ref{def-sigma-tau-central-der}.

Description of $(\sigma,\tau)$-central derivations is quite simple.

The homomorphisms $\varphi_{\mu,\nu}$ to additive group of complex numbers look alike
\begin{equation}
   \varphi_{\mu,\nu}: \begin{pmatrix}
 		1 & a & c\\
 		0 & 1 & b\\
 		0 & 0 & 1
 	\end{pmatrix}
	 \mapsto
	 (\mu a+ \nu b).
\end{equation}

The center of group $H$ (both $(\sigma,\tau)$ and usual by Proposition \ref{prop-center-sigma-tau}) contains elements

$$
	z_r = \begin{pmatrix}
 		1 & 0 & r\\
 		0 & 1 & 0\\
 		0 & 0 & 1
 	\end{pmatrix}.
$$

Consider
 \begin{equation} \label{specialsigmatau}
\tilde{\sigma} = \begin{pmatrix}
 		1 & \sigma_a & \sigma_c\\
 		0 & 1 & \sigma_b\\
 		0 & 0 & 1
 	\end{pmatrix},
 	\tilde{\tau} = \begin{pmatrix}
 		1 & \sigma_a & \tau_c\\
 		0 & 1 & \sigma_b\\
 		0 & 0 & 1
 	\end{pmatrix}.
 \end{equation}

\begin{proposition}
	For each centralizer element $u = \sigma(g)u\tau(g^{-1}) \, \forall g \in G$ there exists $r \in \mathbb{R}$, such that $u = z_{r}$. In the other words, $(\sigma,\tau)$-centralzers and elements from $Z_{(\sigma, \tau)} = Z(G)$ become equal.
\end{proposition}
\begin{proof}
	$$
	u = \begin{pmatrix}
 		1 & u_a & u_c\\
 		0 & 1 & u_b\\
 		0 & 0 & 1
 	\end{pmatrix}, g = \begin{pmatrix}
 		1 & g_a & g_c\\
 		0 & 1 & g_b\\
 		0 & 0 & 1
 		\end{pmatrix},
	$$
	$$
	\sigma(g)u\tau(g^{-1}) = \begin{pmatrix}
 		1 & u_a & u_c + g_au_b - u_bg_a\\
 		0 & 1 & u_b\\
 		0 & 0 & 1
 	\end{pmatrix}
	$$
 	Thus, from the equation $u = \sigma(g)u\tau(g^{-1})$, one gets $u_a = 0$, $u_b = 0$, $u_c = r$.
\end{proof}

The latter equation means that $(\sigma,\tau)$-derivations given by \eqref{specialsigmatau} become equal to the usual derivations considered in \cite{Ar}.

So, we can calculate the $(\sigma,\tau)-$central derivation $d^{r}_{\varphi_{\mu,\nu}}$ on the generator of algebra $\mathcal{H}$:

\begin{equation}
	d^{r}_{\varphi_{\mu,\nu}} \begin{pmatrix}
 		1 & a & c\\
 		0 & 1 & b\\
 		0 & 0 & 1
 	\end{pmatrix} =
 	(\mu a+\nu b)\begin{pmatrix}
 		1 & a & c + \sigma_ab - \sigma_ba + r \\
 		0 & 1 & b\\
 		0 & 0 & 1
 	\end{pmatrix}.
\end{equation}

\section{$(\sigma,\tau)$-$FC$ groups}
\label{sec:sigmatauFCgroups}

The class of $FC$-groups is an interesting class of groups for which conditions are similar to condition of finite groups. More detailed study of $FC$-groups can be found in \cite{Robinson72,Gorch65,Gorch78}.

In classical group theory, $FC$-group is a group in which all conjugacy classes are finite. In this terms abelian group is a group where all conjugacy classes contain one element. These concepts are naturally carried to the case of $(\sigma,\tau)$-groups.

\begin{definition}\begin{itemize}
\item Group $G$ is a $(\sigma,\tau)$-$FC$ group if each $(\sigma,\tau)$-conjugacy class $[u]_{\sigma,\tau}\in G^{(\sigma,\tau)}$ is finite.
\item Group $G$ is a $(\sigma,\tau)$-abelian (or $(\sigma,\tau)-$commutative) if each class $[u]_{\sigma,\tau}$ contains single element.
\end{itemize}
\end{definition}
For $(\sigma,\tau)$-commutative group, the following identity holds
\begin{equation}
	\sigma(v)u = u\tau(v),\;\forall u,v\in G.
\end{equation}

Apparently the definition of $(\sigma,\tau)$-$FC$ group is introduced in this paper for the first time, and so we will show some properties of such groups.
Let us give an example of a source of such groups.

\begin{proposition}
Let endomorphisms $\sigma,\tau$ acting on a group $G$ have a finite image. Then $G$ is a $(\sigma,\tau)$-$FC$ group.
\end{proposition}
\begin{proof}
 If images of endomorphisms $\sigma,\tau$ are finite, then each $(\sigma,\tau)$-conjugacy class is finite by Definition
\ref{def-st-conj-class}.
\end{proof}

Standard $FC$-group may not be an $(\sigma,\tau)$-$FC$ group for arbitrary endomorphisms $(\sigma,\tau)$ even if $\sigma$ and $\tau$ are inner. Let for fixed elements $x,y\in G$, $\sigma_x: g\to xgx^{-1}, \tau_y:g\to yg y^{-1}$. Then for $a\in G$ the corresponding $(\sigma,\tau)$-conjugacy class looks like
\begin{equation}
\label{eq-st-cc}
	[a]_{\sigma_x,\tau_y} = \{xvx^{-1} a yv^{-1}y^{-1}|v\in G\}.
\end{equation}

So if $G$ is an infinite $FC$-group, then we have infinite number of conjugacy classes, so typically there is an infinite number of elements in $(\sigma,\tau)-$conjugacy class. However the following proposition holds.

\begin{proposition}
\label{prop-fc-sigmasigma}
	A group $G$ is  a $(\sigma_x,\sigma_x)$-$FC$ group if and only if $G$ is a $FC$-group.
\end{proposition}
\begin{proof}
For the proof it is enough to see that for $x=y$ formula \eqref{eq-st-cc} can be rewritten in the following way
\begin{equation}
	[a]_{\sigma_x,\sigma_x} = \{xvx^{-1} a xv^{-1}x^{-1}|v\in G\},
\end{equation}
and on the right side we get
$$
xvx^{-1} a xv^{-1}x^{-1} = xvx^{-1} a (xvx^{-1})^{-1},
$$
so elements of $[a]_{\sigma_x,\sigma_x}$ are contained in the usual conjugacy class of the element $a$.
\end{proof}

The following statements follows easily from Proposition \ref{prop-fc-sigmasigma}.
\begin{corollary}
For each $x\in G$
\begin{itemize}
\item  For each group $G$ holds that $[a]_{\sigma_x,\sigma_x} = [a]$, where $[a]$ is the usual conjugacy class.
\item If $G$ is an abelian group then it is $(\sigma_x,\sigma_x)-$abelian.
\end{itemize}
\end{corollary}

\begin{remark}
\label{remark-abelian-st}
	If group $G$ is $(\sigma,\tau)$-abelian it may not be abelian in the usual sense. An example of this is the case when $\sigma=\tau$ and an image of map $\sigma$ subsets in the usual centre of the group $G$.
\end{remark}

Now we will prove an analogue of Theorem \ref{th-decomposition} for $(\sigma,\tau)$-$FC$ groups.

\begin{theorem}
\label{th-FC}
	If $G$ is a finitely generated $(\sigma,\tau)$-$FC$ group, and $\sigma,\tau$ are endomorphisms of group $G$, then
	\begin{equation}
	\label{eq-FC-decomposition-theorem}
	D_{(\sigma,\tau)} \,  \cong \bigoplus\limits_{[a]_{(\sigma,\tau)}} Z^*_{(\sigma,\tau)}(a) \, \bigoplus \, Inn(\Gamma).
\end{equation}
\end{theorem}
Here $Z^*_{(\sigma,\tau)}(a)$ is the space of group characters of the $(\sigma,\tau)$-centralizer $Z_{(\sigma,\tau)}(a)$ as in our Theorem \ref{th-decomposition}.
\begin{proof}
	The proof of \eqref{eq-FC-decomposition-theorem} in our theorem is similar to the proof of Theorem \ref{th-decomposition}. Except for one moment: we have to proof that all quasi-inner $(\sigma,\tau)-$derivations for the case of $(\sigma,\tau)-$$FC$ group are inner.
	
For quasi-inner $(\sigma,\tau)-$derivations, the Theorem \ref{th-qinn-calc} is applicable.
The set of objects in each subgroupoid $\Gamma_{[u]_{\sigma,\tau}}$ is finite. So the following formula holds:
\begin{equation}
\label{eq-d-1}
	d(g) = \sum\limits_{h \in G} (P(h\tau(g^{-1})) - P(\sigma(g^{-1})h))h.
\end{equation}

From formula \eqref{eq-decomp-char-supports} we have the following decomposition for derivation $d$:
\begin{equation}
\label{eq-d-decomp}
	d = \sum\limits_{[u]_{\sigma,\tau}\in G^{(\sigma,\tau)}} d_{[u]_{\sigma,\tau}},
\end{equation}
where derivation $d_{[u]_{\sigma,\tau}}$ is supported in groupoid $\Gamma_{[u]_{\sigma,\tau}}$. First we will prove that each term is inner and then check that sum \eqref{eq-d-decomp} is finite.

Consider the fixed term $d_{[u]_{\sigma,\tau}}$. The set of objects in $\Gamma_{[u]_{\sigma,\tau}}$ is finite so  the right side in formula \eqref{eq-d-1} is nonzero just for $h\in [u]_{\sigma,\tau}$. That implies that derivation $d_{[u]_{\sigma,\tau}}$ is inner and holds the formula
\begin{equation}
	d_{[u]_{\sigma,\tau}}(g) = [\sum\limits_{h\in {[u]_{\sigma,\tau}}} P(h)h, g].
\end{equation}

If $G$ is a finite or abelian group proof of innerness of derivation $d$ is trivial. So let $G$ be an infinite group.

By assumption $G=<g_1,\dots,g_n>$ is a finitely generated group. If the support of character $\chi$ contains infinite number of subgroupoids, then for some $i\in 1..n$ character is not trivial on infinite number of morphisms of the form $(\ast,g_i)$, which contradicts with locally finiteness condition which is necessary for $\chi$ to yield the derivation.
\end{proof}

\begin{corollary}
\label{prop-finite-case}
	If $G$ is a finite group then all $(\sigma,\tau)$-derivations are inner.
\end{corollary}

\begin{proof}
If $G$ is a finite group, then the set morphisms in our groupoid $\Gamma$ is finite because the set $Hom(\Gamma)$ is a Cartesian product $G\times G$. So the set of loops around each object is finite. But if $\zeta\in Hom(a,a)$ and $\chi(\zeta)\neq 0$ then $\chi(\zeta^n) \neq 0$, so the set of loops $\{\zeta^n|n\in \mathbf{N}\}$ is infinite which is impossible.

That gives us triviality of the character $\chi$ on all loops and it remains to apply the Theorem \ref{th-FC}.
\end{proof}

The general case when maps $\sigma,\tau$ are endomorphisms of group algebra was reviewed in
\cite{ChaudhuriComAlg2019}. In the cited paper, the following theorem (see Theorem 1.1) was proved.

\begin{theorem}[Chaudhuri, 2019]
 Let $G$ be a finite group and $R$ be an integral domain with $1$ with characteristic $p\geq0$ such that $p$ does not divide the order of $G$.
  \begin{enumerate}
  \item If $R$ is a field and $\sigma$, $\tau$ are algebra endomorphisms of $RG$ such that they fix $\mathcal{Z}(RG)$ elementwise, then every $(\sigma,\tau)$-derivation of $RG$ is $(\sigma,\tau)$-inner.
  \item If $R$ is an integral domain that is not a field and $\sigma,\;\tau$ are $R$-linear extensions of group homomorphisms of $G$ such that they fix $\mathcal{Z}(RG)$ elementwise, then every $(\sigma,\tau)$-derivation of $RG$ is $(\sigma,\tau)$-inner.
  \end{enumerate}
\end{theorem}

Note that if $\sigma$ and $\tau$ are identical isomorphisms, then we get a well-known theorem that in group algebras for finite groups all derivations are inner.

Another natural application of Theorem \ref{th-FC} is a case of $(\sigma,\tau)$-abelian group. It is easy to see that in $(\sigma,\tau)$-abelian groups the $(\sigma,\tau)$-commutator is trivial, so there are no inner
$(\sigma,\tau)$-derivations.

\begin{corollary}
 The ${\sigma,\tau}$-derivation algebra $ D_{(\sigma,\tau)}$ of $(\sigma,\tau)$-abelian group coincides with $(\sigma,\tau)$-central derivations.
\end{corollary}
\begin{proof}
	It is easy to see that derivation is $(\sigma,\tau)$-central if and only if it's support contains single element.
\end{proof}

Considering Remark \ref{remark-abelian-st} we note that the case when $G$ is abelian group is significantly different to the case of $(\sigma,\tau)$-abelian group. The case of $\sigma$-derivations for abelian groups was studied in \cite{SongWuXin2015}.

Now we can find criterion of innerness of $(\sigma,\tau)$-derivations in $(\sigma,\tau)-$$FC$ group which is similar to Corollary~\ref{coroll-criterion-triv}.

\begin{corollary}
	For groups obeying conditions of the Theorem \ref{th-FC}, all $(\sigma,\tau)$-derivations are inner if and only if the following condition is satisfied: all $(\sigma,\tau)-$centralizers are such that $\forall a\in G$, the factor-group $Z_{(\sigma,\tau)}(a)/ Z'_{(\sigma,\tau)}(a)$ is a periodic group.
\end{corollary}

Here $Z'_{(\sigma,\tau)}(a)$ is a commutant of a group $Z_{(\sigma,\tau)}(a)$.

\begin{proof}
The prove is similar to prove of Corollary \ref{coroll-criterion-triv}.
	The group $Z_{(\sigma,\tau)}(a)/Z'_{(\sigma,\tau)}(a)$ is naturally abelian. So locally finiteness is equivalent to triviality for each $a\in G$ of the space of group characters $Z_{(\sigma,\tau)}(a)/Z'_{(\sigma,\tau)}(a)$. So from the Theorem \ref{th-FC} we get that all $(\sigma,\tau)$-derivations are quasi-inner.
\end{proof}

\section*{Acknowledgments.}  Andronick Arutyunov was supported the Grants Council of the President of the Russian Federation, project no.~MK-2364.2020.1. The results in Subsection \ref{subsec-Heis} were obtained under the support of the Ministry of Science and Higher Education of the Russian Federation (Goszadaniye 075-00337-20-03, project no. 0714-2020-0005). The results of Aleksandr Alekseev in Subsection \ref{subsec:inneendoms} were obtained with the support of the Russian Science Foundation (grant no. 20-11-20131) in V.A. Trapeznikov Institute of Control Sciences of RAS.

Authors are grateful to Professor Alexander Mishchenko for interest to this work.

\end{document}